\documentclass[a4paper,11pt]{article}

\usepackage{amsmath,amsthm,amssymb,mathrsfs,latexsym,amsfonts}
\usepackage{graphicx,psfrag,epsfig}
\usepackage{bbm}
\usepackage[english]{babel}
\usepackage[latin1]{inputenc}
\usepackage{a4wide}

\newtheorem{theorem}{Theorem}

\newtheorem{proposition}[theorem]{Proposition}

\newtheorem{question}{Question}

\newcounter{paraga}[section]
\renewcommand{\theparaga}{{\bf\arabic{paraga}.}}
\newcommand{\paraga}{\medskip \addtocounter{paraga}{1} 
\noindent{\theparaga\ } }

\begin{document}

\def\MP{\,{<\hspace{-.5em}\cdot}\,}
\def\SP{\,{>\hspace{-.3em}\cdot}\,}
\def\PM{\,{\cdot\hspace{-.3em}<}\,}
\def\PS{\,{\cdot\hspace{-.3em}>}\,}
\def\EP{\,{=\hspace{-.2em}\cdot}\,}
\def\PP{\,{+\hspace{-.1em}\cdot}\,}
\def\PE{\,{\cdot\hspace{-.2em}=}\,}
\def\N{\mathbb N}
\def\C{\mathbb C}
\def\Q{\mathbb Q}
\def\R{\mathbb R}
\def\T{\mathbb T}
\def\A{\mathbb A}
\def\Z{\mathbb Z}
\def\demi{\frac{1}{2}}

\begin{titlepage}
\author{Abed Bounemoura~\footnote{Email: abedbou@gmail.com. CNRS-IMPA, UMI 2924.}}
\title{\LARGE{\textbf{Instability of resonant invariant tori in Hamiltonian systems}}}
\end{titlepage}

\maketitle

\begin{abstract}
In this article, we consider the dynamics in a neighborhood of a quasi-periodic torus which is invariant by a Hamiltonian flow, we discuss several notions of stability and we prove several results of instability when the frequency of the invariant torus is resonant.
\end{abstract}

\section{Introduction and main results}

\paraga Let $n \geq 2$ and $\T^n:=\R^n/\Z^n$. Consider a Hamiltonian system on $\T^n \times \R^n$ associated to a $C^l$, $l \geq 2$, function of the form
\begin{equation}\label{Ham1}
H(\theta,I)=\omega \cdot I + A(\theta)I\cdot I+R(\theta,I), \quad (\theta,I)\in \T^n \times \R^n  
\end{equation}
where $\cdot$ denotes the Euclidean inner product, $\omega \in \R^n$, $A: \T^n \rightarrow \mathrm{Sym}(n,\R)$ is a $C^l$ map taking values in the space of real symmetric matrices of size $n$, and $R(\theta,I)=O_3(I)$ is of order at least $3$ in $I$. The set $\mathcal{T}_\omega:=\T^n \times \{I=0\}$ is invariant by the Hamiltonian flow of $H$, it is a Lagrangian quasi-periodic torus with frequency $\omega$, and any such torus (on an arbitrary symplectic manifold) is of this form. 

Assuming that the frequency $\omega$ is non-resonant, that is, $k\cdot \omega\neq 0$ for any $k\in\Z^n \setminus \{0\}$, the invariant torus $\mathcal{T}_\omega$ possesses, generically, some stability properties. First, if $H$ is smooth and the invariant torus is Kolmogorov non-degenerate (that is, the symmetric matrix $A_0:=\int_{\T^n}A(\theta)d\theta$ is non-singular), then it is \emph{KAM stable}: in any sufficiently small neighborhood of $\mathcal{T}_\omega$, there is a set of smooth Lagrangian quasi-periodic invariant tori which has positive Lebesgue measure and density one at $\mathcal{T}_\omega$, and which are obtained by a small deformation of the unperturbed invariant tori (those associated to the linear part of the Hamiltonian~\eqref{Ham1}). This result was first proved in \cite{EFK}, under extra assumptions on $\omega$ and $H$ (namely that $\omega$ has finite uniform Diophantine exponent and $H$ is real-analytic), and the general case was subsequently obtained in \cite{Bou14}. A consequence of the KAM stability is the following stability property: for any $\varepsilon>0$ sufficiently small, most solutions $(\theta(t),I(t))$ for which $I(0)$ has a norm equals to $\varepsilon$ satisfy
\begin{equation}\label{stabKAM}
|I(t)-I(0)| \leq \varepsilon\sqrt{\mu(\varepsilon)}, \quad t \in \R 
\end{equation}
where $\mu(\varepsilon)$ is a function going to zero as $\varepsilon$ goes to zero (this function depends on the arithmetic properties of $\omega$, see \cite{Bou14} or the definition~\eqref{mu} below). 

Still assuming $\omega$ non-resonant and without further assumptions, the invariant torus is also \emph{effectively stable} (or Nekhoroshev stable): for all solutions one has the estimate
\begin{equation}\label{stabNek}
|I(t)-I(0)| \leq \varepsilon\mu(\varepsilon), \quad |t| \leq T(\varepsilon)
\end{equation}
where $T(\varepsilon)$ is at least of order $\varepsilon^{-1}$. More precisely, if $H$ is smooth then $T(\varepsilon)$ is of order $\varepsilon^{-1}\mu(\varepsilon)^{-l}$ for any fixed $l \in \N$, and if $H$ is real-analytic (or Gevrey smooth), this latter power estimate can be replaced by an exponential estimate. Proofs can be found in~\cite{Bou12}, \cite{Bou13a} and~\cite{Bou13b}. In fact, in view of the results in~\cite{BFN14}, one may expect that generically, the above exponential stability estimate can be improved to a double exponential estimate.   

Finally, there is another more classical notion of stability: the invariant torus is \emph{topologically stable} (or Lyapounov stable) if it has a basis of invariant neighborhoods. Observe that for $n=2$, if the invariant torus is KAM stable on each energy level sufficiently close to the zero energy level (which contains the invariant torus), then it is topologically stable, but in general this notion is not related to KAM stability or effective stability. Extrapolating on conjectures of Arnold (see \cite{Arn94}), one expects that generically, for $n \geq 3$, the torus is topologically unstable. However, essentially nothing is known (in the smooth case one only knows how to construct examples, see \cite{Dou88}) and absolutely nothing is known in the real-analytic case (no examples are known).

\paraga  Our aim here is to investigate the case where one drops the requirement that the frequency $\omega$ is non-resonant. More precisely, our aim is to prove instability results for a generic Hamiltonian as in~\eqref{Ham1}, assuming that the frequency $\omega$ is resonant. 

To state our results, we start by assuming that the vector $\omega$ is of the form 
\begin{equation}\label{omega}
\omega=(0,\tilde{\omega}) \in \R^d \times \R^{n-d}
\end{equation}
for some $1 \leq d \leq n-1$ and some non-resonant vector $\tilde{\omega} \in \R^{n-d}$. This is no loss of generality, as by a linear symplectic transformation one can always put $\omega$ in this form. Up to a constant time change, we may also assume that $|\omega|=|\tilde{\omega}|=1$, where $|\,.\,|$ denotes the supremum norm.

Next we define, for $Q \geq 1$, the function $\Psi=\Psi_\omega$ by
\[ \Psi(Q):=\max\{|k\cdot\tilde{\omega}|^{-1} \; | \; k=(k_1,\dots,k_{n-d}) \in \Z^{n-d}, \; 0< |k_1|+\cdots+|k_{n-d}|\leq Q \} \] 
and for $x\geq \Psi(1)=|\tilde{\omega}|^{-1}=1$, the function $\Delta=\Delta_\omega$ by
\[ \Delta(x):=\sup\{Q \geq 1 \; | \; Q\Psi(Q)\leq x\}. \]
For some constant $\kappa$ which depends only on $n$, we define, given a small parameter $\varepsilon>0$, another small parameter
\begin{equation}\label{mu}
\mu(\varepsilon):=\left(\Delta\left(\kappa\varepsilon^{-1}\right)\right)^{-1}.
\end{equation}
Observe that $\mu(\varepsilon)$ always converge to zero as $\varepsilon$ goes to zero, more slowly than $\varepsilon$: for instance, if $\omega$ is periodic (a multiple of a rational vector), then $\mu(\varepsilon)$ is exactly of order $\varepsilon$, and if $\omega$ is resonant-Diophantine (meaning that it is not rational but the function $\Psi$ defined above grows at most as a power), then $\mu(\varepsilon)$ is of order a power of $\varepsilon$. In general, the speed of convergence to zero can be arbitrarily slow.

Recall that we are considering $H$ as in~\eqref{Ham1} in a small neighborhood of the origin, which, without loss of generality, will be chosen to be the domain
\begin{equation*}
\T^n \times B_{3\varepsilon}:=\T^n \times \{I\in \R^n  \; | \; |I|<3\varepsilon \}
\end{equation*}
for $\varepsilon>0$. Letting
\begin{equation}\label{scale1}
I:=\varepsilon \tilde{I}, \quad H(\theta,I):=\varepsilon\tilde{H}(\theta,\tilde{I}),
\end{equation}
it is then equivalent to consider the Hamiltonian
\begin{equation}\label{Ham2}
\tilde{H}(\theta,\tilde{I})=\varepsilon^{-1}H(\theta,I)=\varepsilon^{-1}H(\theta,\varepsilon \tilde{I})=\omega \cdot \tilde{I} + \varepsilon A(\theta)\tilde{I}\cdot \tilde{I}+\varepsilon^2 \tilde{R}(\theta,\tilde{I}) 
\end{equation}
on the domain $\T^n \times B_3$, with the estimates
\begin{equation}\label{bounds}
|A|_{C^l(\T^n)} \leq C_1, \quad |\tilde{R}|_{C^l(\T^n \times B_3)} \leq C_2.
\end{equation}
Here, $|\,.\,|_{C^l(\T^n \times B_3)}$ denotes the usual $C^l$-norm on the domain $\T^n \times B_3$, and $|\,.\,|_{C^l(\T^n)}$ the usual $C^l$-norm on $\T^n$. 

Let us now define the map $\bar{A} : \T^d \rightarrow \mathrm{Sym}(n,\R)$ by
\begin{equation}\label{Abar}
\bar{A}(\theta_1,\dots,\theta_d):=\int_{\T^{n-d}}A(\theta_1,\dots,\theta_d,\theta_{d+1},\dots,\theta_n)d\theta_{d+1}\dots d\theta_n.
\end{equation}
Obviously, $\bar{A}$ is of class $C^l$ and $|A|_{C^l(\T^d)} \leq C_1$. Our first generic assumption is as follows.

\bigskip

$(A.1)$ The function $\bar{A} : \T^d \rightarrow \mathrm{Sym}(n,\R)$ is non-constant. Hence there exist $\theta^* \in \T^d$ and an integer $i=i(\theta^*)$, $1 \leq i \leq d$, such that $A_i^*:=\partial_{\theta_i}\bar{A}(\theta^*) \in \mathrm{Sym}(n,\R)$ is non-zero.

\bigskip

This assumption $(A.1)$ is important to derive instability results; in Theorem~\ref{mainthm4} below we will see that if $(A.1)$ is not satisfied, then generically the invariant torus is KAM stable.

Now given the assumption $(A.1)$, let us define 
\[ C_i^*:=\{v \in \R^n \; | \; A_i^*v\cdot v=0\}, \]
that is $C_i^*$ is the isotropic cone of the quadratic form associated to the matrix $A_i^*$. Since $A_i^*$ is non-zero, the complement of $C_i^*$ in $\R^n$ is open, dense and of full Lebesgue measure.  

\paraga We can finally state our first result.

\begin{theorem}\label{mainthm1}
Let $H$ be as in~\eqref{Ham1} satisfying~\eqref{bounds} with $l\geq 3$ and with $\omega$ as in~\eqref{omega}. Assume that $(A.1)$ is satisfied, and fix $\tilde{I}^0 \in \R^n \setminus C_i^*$ with $|\tilde{I}^0|=1$. There exist positive constants $\mu_0=\mu_0(n,C_1,C_2,A_i^*,\tilde{I}^0)$, $c=c(n,C_1,A_i^*,\tilde{I}^0)$ and $\delta=\delta(n,C_1,A_i^*,\tilde{I}^0)$ such that if 
\[ 0<\mu(\varepsilon)\leq \mu_0, \]
then there exists a solution $(\theta(t),I(t))$ of the Hamiltonian system associated to~\eqref{Ham1} such that
\[ I(0)=\varepsilon \tilde{I^0}, \quad |I(\tau)-I(0)|\geq |I_i(\tau)-I_i(0)|\geq c\varepsilon, \quad \tau:=\delta\varepsilon^{-1}.  \] 
Moreover, there exists a positive constant $C=c(n,C_1,C_2,,A_i^*,\tilde{I}^0)$ such that
\[ \max_{d+1 \leq j \leq n}|I_j(\tau)-I_j(0)|\leq C\varepsilon\mu(\varepsilon).  \]
\end{theorem} 

Theorem~\ref{mainthm1} clearly shows that generically, the invariant torus is not effectively stable as the estimate~\eqref{stabNek} cannot be satisfied for all solutions. It also shows that generically it is not KAM stable. Indeed, the statement implies the existence of ``many" solutions for which $I(0)$ has a norm equals to $\varepsilon$ but yet $|I(\tau)-I(0)|\geq c\varepsilon$, for some positive time $\tau$ and some constant $c$ independent of $\varepsilon$, provided the latter is sufficiently small. By ``many" it is meant that we can find such solution for any choice of $\varepsilon\tilde{I}^0$, provided $\tilde{I}^0 \notin C_i^*$ and $\mu(\varepsilon)\leq \mu_0$ where $\mu_0$ depends on $\tilde{I}^0$: on any sufficiently small ball around the origin, the set of initial action that are not concerned with Theorem~\ref{mainthm1} consists of small cusps around directions associated to $C_i^*$, and the relative measure of this set goes to zero as the radius of the ball goes to zero. Hence the consequence~\eqref{stabKAM} of KAM stability cannot be true, and this shows that under our assumption $(A.1)$, the invariant torus cannot be KAM stable.  

\paraga Theorem~\ref{mainthm1} implies the existence of solutions, starting arbitrarily close to the invariant torus, and for which the action variables deviate from its initial condition. Yet it does imply the existence of solutions for which the action variables deviate from zero (which corresponds to the invariant torus), or equivalently, solutions for which the action variables increase in norm. Under an extra generic assumption, this can be achieved. Let $e_i$ be the $i$-th vector of the canonical basis of $\R^n$, where the index $i$ is the one given by $(A.1)$.

\bigskip

$(A.2)$ The vector $e_i$ does not belong to $C_i^*$, that is $A_i^*e_i\cdot e_i:=a_i^*\neq 0$. Reversing time if necessary, we may assume that $a_i^*<0$.

\bigskip

Here's our second result.

\begin{theorem}\label{mainthm2}
Let $H$ be as in~\eqref{Ham1} satisfying~\eqref{bounds} with $l\geq 3$ and with $\omega$ as in~\eqref{omega}. Assume that $(A.1)$ and $(A.2)$ are satisfied. There exist positive constants $\mu_0=\mu_0(n,C_1,C_2,a_i^*)$, $c=c(n,C_1,a_i^*)$ and $\delta=\delta(n,C_1,a_i^*)$ such that if 
\[ 0<\mu(\varepsilon)\leq \mu_0, \]
then there exists a solution $(\theta(t),I(t))$ of the Hamiltonian system associated to~\eqref{Ham1} such that
\[ I(0)=\varepsilon e_i, \quad |I(\tau)|\geq (1+c)\varepsilon=(1+c)|I(0)|, \quad \tau:=\delta\varepsilon^{-1}.  \]
Moreover, there exists a positive constant $C=c(n,C_1,C_2,a_i^*)$ such that
\[ \max_{d+1 \leq j \leq n}|I_j(\tau)-I_j(0)|\leq C\varepsilon\mu(\varepsilon).  \] 
\end{theorem} 

This theorem implies in particular that given any sufficiently small ball $B$ around $0 \in \R^{n}$, the domain $\T^n \times B$ cannot be invariant by the Hamiltonian flow. However, this is still far from proving that the invariant torus is not topologically stable. The statement of Theorem~\ref{mainthm2} (and of Theorem~\ref{mainthm1} also) only gives information on the behavior of solutions for a finite time scale of order $\varepsilon^{-1}$. To prove topologically instability, one needs to be able to control solutions over infinite interval of times, and this is a non-trivial task. The following question is therefore still open.

\begin{question}
Consider a Hamiltonian $H$ as in~\eqref{Ham1} and assume that $\omega$ is resonant. Is is true that, generically, the invariant torus $\mathcal{T}_\omega$ is topologically unstable? 
\end{question}

One expects that the answer to this question is positive, even without the assumption that $\omega$ is resonant. As explained in the Introduction, the case where $\omega$ is non-resonant is much more complicated. One only knows examples of topologically unstable tori in the smooth category; in the analytic category, examples are not known, even examples with the weaker instability property expressed in Theorem~\ref{mainthm2}.

If $\omega$ is resonant, examples are easily constructed, even in the analytic case. For instance, consider the Hamiltonian
\[ H(\theta,I)=\omega \cdot I +A(\theta_1)I\cdot I \]
where $\omega=(0,\tilde{\omega})\in \R^d \times \R^{n-d}$ with $\tilde{\omega}$ non-resonant, $A(\theta_1)$ is a diagonal matrix of the form $\mathrm{Diag}(a_1(\theta_1),\dots,a_n(\theta_1))$ with functions $a_i(\theta_1)$, $1 \leq i \leq n$, chosen such that $a_1$ is identically zero and $a_j'(\theta_1^*)>0$ for some $\theta_1^* \in \T$ and some $2 \leq j \leq n$. Fix the initial angle to be $\theta(0)=(\theta_1^*,\theta_2(0),\dots,\theta_n(0))$ with $(\theta_2(0),\dots,\theta_n(0)) \in \T^{n-1}$ arbitrary. Choosing any non-zero initial action proportional to $e_j$, one easily sees that the corresponding solution is defined for all time and that the first action component is unbounded, and hence the invariant torus is not topologically stable. 

Coming back to Theorem~\ref{mainthm1}, we argued that it implies that the invariant torus cannot be KAM stable. Yet, for the same reason as we are not able to decide if topological instability holds in general (which is that we only control the solution on a finite time scale), we cannot decide also if the invariant torus is accumulated by other quasi-periodic invariant tori.  

\begin{question}
Consider a Hamiltonian $H$ as in~\eqref{Ham1} and assume that $\omega$ is resonant. Is the invariant torus $\mathcal{T}_\omega$ generically accumulated by other quasi-periodic invariant tori or is it generically isolated? 
\end{question}

Looking at the example we give above, it is easy to observe that any initial action proportional to $e_1$ leads to a quasi-periodic invariant tori. Hence even the much simpler question below is left unanswered.

\begin{question}
Construct a Hamiltonian $H$ as in~\eqref{Ham1} with $\omega$ resonant for which the invariant torus $\mathcal{T}_\omega$ is isolated.
\end{question}

\paraga Coming back again to Theorem~\ref{mainthm1}, one drawback of the statement is that one has to exclude small cusps around isotropic directions of the symmetric matrix $A_i^*$. If the latter is empty, one obtains a uniform statement. Let us introduce another condition.
   
\bigskip

$(A.3)$ The symmetric matrix $A_i^*$ is sign-definite. Reversing time if necessary, we may assume that it is negative definite, hence there exists a negative constant $\lambda_i^*$ such that for all $I\in \R^n$, $A_i^*I\cdot I \leq \lambda_i^*|I|^2$. 

\bigskip

This condition is obviously much stronger than the conditions $(A.1)$ and $(A.2)$ together. Even though this condition is open, it is clearly no more generic. 

Our third result is as follows.  

\begin{theorem}\label{mainthm3}
Let $H$ be as in~\eqref{Ham1} satisfying~\eqref{bounds} with $l\geq 3$ and with $\omega$ as in~\eqref{omega}. Assume that $(A.3)$ is satisfied. There exist positive constants $\mu_0=\mu_0(n,C_1,C_2,\lambda_i^*)$, $c=c(n,C_1,\lambda_i^*)$ and $\delta=\delta(n,C_1,\lambda_i^*)$ such that if 
\[ 0<\mu(\varepsilon)\leq \mu_0, \]
then, given any $\tilde{I}_0 \in \R^n$ such that $|\tilde{I}_0|=1$, there exists a solution $(\theta(t),I(t))$ of the Hamiltonian system associated to~\eqref{Ham1} such that
\[ I(0)=\varepsilon \tilde{I}_0, \quad |I(\tau)-I(0)|\geq I_i(\tau)-I_i(0)\geq c\varepsilon, \quad \tau:=\delta\varepsilon^{-1}.  \] 
Moreover, there exists a positive constant $C=c(n,C_1,C_2,\lambda_i^*)$ such that
\[ \max_{d+1 \leq j \leq n}|I_j(\tau)-I_j(0)|\leq C\varepsilon\mu(\varepsilon).  \] 
\end{theorem} 

\paraga Finally, as we already mentioned, let us discuss briefly what happens if the condition $(A.1)$ is not satisfied. More precisely, consider the following assumption.

\bigskip

$(A.4)$ There exists a non-singular matrix $A_0 \in \mathrm{Sym}(n,\R)$ such that for all $(\theta_1,\dots,\theta_d)\in \T^d$, $\bar{A}(\theta_1,\dots,\theta_d)=A_0$.

\bigskip

Under the following assumption, one can prove, as in~\cite{Bou14}, that the invariant torus is KAM stable provided the Hamiltonian is sufficiently smooth.

Here's our last result. 

\begin{theorem}\label{mainthm4}
Let $H$ be as in~\eqref{Ham1} satisfying~\eqref{bounds} with $l\geq l_0+1>3n$ and with $\omega$ as in~\eqref{omega}. Assume that $(A.4)$ is satisfied. Then there exist positive constants $\mu_0=\mu_0(n,C_1,C_2,A_0)$, $c_1=c(n,C_1,C_2,A_0)$ and $c_2=c_2(n,C_1,C_2,A_0)$ such that if
\[ 0 < \mu(\varepsilon) \leq \mu_0,\]
there exists a set $\mathcal{K} \subset \T^n \times B_{2\varepsilon}$, which consists of Lagrangian quasi-periodic tori invariant by the Hamiltonian flow of $H$. Moreover, each tori is of class $C^{l_0'}$, for $l_0'<l_0-2n+1$, and we have the measure estimate
\[ c_1\sqrt{\mu(\varepsilon)}\mathrm{Leb}(\T^n \times B_{2\varepsilon}) \leq \mathrm{Leb}(\T^n \times B_{2\varepsilon} \setminus \mathcal{K}) \leq  c_{2}\sqrt{\mu(\varepsilon)}\mathrm{Leb}(\T^n \times B_{2\varepsilon}). \]
\end{theorem} 

Again, as in~\cite{Bou14}, if the Hamiltonian is more regular (smooth, Gevrey smooth or real-analytic), then the invariant tori found are as regular as the Hamiltonian. But at variance with~\cite{Bou14}, one does not get a better measure estimate if the Hamiltonian is Gevrey smooth or real-analytic.

\section{Proof of the results}

\paraga The proof of Theorem~\ref{mainthm1} follows the strategy of \cite{BK14}. The essential step in the argument is the construction of a global resonant normal, which was done in \cite{Bou13a} and that we recall below. Proofs of Theorem~\ref{mainthm2} and Theorem~\ref{mainthm3} proceed exactly the same way. Finally, the proof of Theorem~\ref{mainthm1} is analogous to the proof of the main theorem of~\cite{Bou14}.

\paraga Recall that, up to the scalings~\eqref{scale1}, the Hamiltonian~\eqref{Ham1} on the domain $\T^n \times B_{3\varepsilon}$ is equivalent to the Hamiltonian~\eqref{Ham2} on the domain $\T^n \times B_{3}$, so let us consider the latter. We define $f=f_\varepsilon$ by
\[ f(\theta,\tilde{I}):=A(\theta)\tilde{I}\cdot \tilde{I}+\varepsilon \tilde{R}(\theta,\tilde{I}) \]
so that the Hamiltonian~\eqref{Ham2} can be written as
\begin{equation}\label{Ham3}
\tilde{H}(\theta,\tilde{I})=\omega \cdot \tilde{I} + \varepsilon f(\theta,\tilde{I}).
\end{equation}
Let us also define the average, with respect to the resonant linear flow associated to $\omega=(0,\tilde{\omega})\in \R^d \times \R^{n-d}$, of the perturbation by
\begin{eqnarray*}
\bar{f}(\theta_1,\dots,\theta_d,\tilde{I}) 
& := & \int_{\T^{n-d}}f(\theta_1,\dots,\theta_d,\theta_{d+1},\dots,\theta_n,\tilde{I})d\theta_{d+1}\dots d\theta_{n} \\
& = & \bar{A}(\theta_1,\dots,\theta_d)\tilde{I}\cdot \tilde{I} +\varepsilon \bar{R}(\theta_1,\dots,\theta_d,\tilde{I}).
\end{eqnarray*}

The Hamiltonian~\eqref{Ham2}, written as in~\eqref{Ham3}, is an $\varepsilon$-perturbation of a linear integrable Hamiltonian, and we have the following resonant normal form.

\begin{proposition}\label{normalform}
There exist $\mu_*=\mu_*(n,C_1,C_2)$ and $C_*=C_*(n,C_1,C_2)$ such that if
\[ 0<\mu(\varepsilon)\leq \mu_*, \]
then there exists a symplectic embedding $\Phi : \T^n \times B_2 \rightarrow \T^n \times B_3$ of class $C^{l-1}$ such that
\[ \tilde{H} \circ \Phi (\theta,\tilde{I}) = \omega \cdot \tilde{I} +\varepsilon \bar{f}(\theta_1,\dots,\theta_d,\tilde{I})+ \varepsilon\mu(\varepsilon)\tilde{f}(\theta,\tilde{I}) \]
with the estimates $|\Phi-\mathrm{Id}|_{C^{l-1}(\T^n \times B_2)} \leq C_*\mu(\varepsilon)$ and $|\tilde{f}|_{C^{l-1}(\T^n \times B_2)} \leq C_*$.
\end{proposition}

This is a special case of Theorem $1.1$ of \cite{Bou13b}, to which we refer for a proof. 

\paraga The proof of Theorem~\ref{mainthm1} will follow from direct arguments using the normal form Proposition~\ref{normalform}.

\begin{proof}[Proof of Theorem~\ref{mainthm1}]
Under the assumptions of Theorem~\ref{mainthm1}, it is sufficient to prove that there exist positive constants $\mu_0=\mu_0(n,C_1,C_2,A_i^*,\tilde{I}^0)$, $c=c(n,C_1,A_i^*,\tilde{I}^0)$, $\delta=\delta(n,C_1,A_i^*,\tilde{I}^0)$ and $C=C(n,C_1,C_2,A_i^*,\tilde{I}^0)$ such that if 
\[ 0<\mu(\varepsilon)\leq \mu_0, \]
then there exists a solution $(\theta(t),\tilde{I}(t))$ of the Hamiltonian system associated to~\eqref{Ham2} such that
\[ \tilde{I}(0)=\tilde{I^0}, \quad |\tilde{I}(\tau)-\tilde{I}(0)|\geq |\tilde{I}_i(\tau)-\tilde{I}_i(0)|\geq c, \quad \tau:=\delta\varepsilon^{-1} \] 
and
\[ \max_{d+1 \leq j \leq n}|\tilde{I}_j(\tau)-\tilde{I}_j(0)|\leq C\mu(\varepsilon).  \]
Indeed, in view of the scalings~\eqref{scale1}, this solution $(\theta(t),\tilde{I}(t))$ of the Hamiltonian system associated to~\eqref{Ham2} gives a solution $(\theta(t),I(t))=(\theta(t),\varepsilon\tilde{I}(t))$ of the Hamiltonian system associated to~\eqref{Ham1} which satisfies the conclusions of Theorem~\ref{mainthm1}.

So let us prove the above statement for the Hamiltonian~\eqref{Ham2}, and to simplify notations, we remove the tilde so that the Hamiltonian~\eqref{Ham2} now reads
\[ H(\theta,I)=\omega\cdot I+\varepsilon A(\theta)I\cdot I+\varepsilon^2R(\theta,I)=\omega\cdot I+\varepsilon f(\theta,I). \] 
Fix $I_0 \in \R^n$ such that $|I_0|=1$ and $A_i^*I_0\cdot I_0 \neq 0$. Thus we can find a positive constant $\gamma_i^*=\gamma_i^*(I_0,A_i^*)$ such that 
\begin{equation}\label{min}
|A_i^*I_0\cdot I_0|=|\partial_{\theta_i}A(\theta^*)I_0\cdot I_0| \geq 4\gamma_i^*. 
\end{equation}
We first choose $\mu_0 \leq \mu_*$ so that Proposition~\ref{normalform} can be applied: it gives the existence of a symplectic embedding $\Phi : \T^n \times B_2 \rightarrow \T^n \times B_3$ of class $C^{l-1}$ such that
\[ H \circ \Phi (\theta,I) = \omega \cdot I +\varepsilon \bar{f}(\theta_1,\dots,\theta_d,I)+ \varepsilon\mu(\varepsilon)\tilde{f}(\theta,I) \]
with the estimates 
\begin{equation}\label{estdist}
|\Phi-\mathrm{Id}|_{C^{l-1}(\T^n \times B_2)} \leq C_*\mu(\varepsilon)
\end{equation}
and 
\begin{equation*}
|\tilde{f}|_{C^{l-1}(\T^n \times B_2)} \leq C_*.
\end{equation*}
Recalling that
\[ \bar{f}(\theta_1,\dots,\theta_d,I) 
=\bar{A}(\theta_1,\dots,\theta_d)I\cdot I +\varepsilon \bar{R}(\theta_1,\dots,\theta_d,I)
\] 
and that $\varepsilon \leq \mu(\varepsilon)$, if we set
\[ f'(\theta,I):=\tilde{f}(\theta,I)+\varepsilon\mu(\varepsilon)^{-1}\bar{R}(\theta_1,\dots,\theta_d,I) \]
then
\[ H \circ \Phi (\theta,I) = \omega \cdot I +\varepsilon \bar{A}(\theta_1,\dots,\theta_d)I\cdot I+ \varepsilon\mu(\varepsilon)f'(\theta,I) \]
with
\begin{equation}\label{estrem}
|f'|_{C^{l-1}(\T^n \times B_2)} \leq C':=C_*+C_2.
\end{equation}

Let us consider the solution $(\theta(t),I(t))$ of the Hamiltonian system associated to $H \circ \Phi$, with initial condition $(\theta(0),I(0))=\Phi(\theta^*,I_0)$, where, abusing notations, we identify $\theta^* \in \T^d$ with $(\theta^*,0) \in \T^n=\T^d \times \T^{n-d}$. Using the estimate~\eqref{estdist}, and choosing $\mu_0$ sufficiently small with respect also to $\gamma_i^*$, one can ensure that 
\begin{equation}\label{min2}
|\partial_{\theta_i}A(\theta(0))I(0)\cdot I(0)| \geq 3\gamma_i^*. 
\end{equation}
The solution $(\theta(t),I(t))$ satisfies, for any $1 \leq l \leq d$ and $d+1 \leq j \leq n$, the equations
\begin{equation}\label{mouve}
\begin{cases}
\dot{I}_l(t)=-\varepsilon\partial_{\theta_l}A(\theta_1(t),\dots,\theta_d(t))I(t)\cdot I(t) -\varepsilon\mu(\varepsilon)\partial_{\theta_l}f'(\theta(t),I(t)), \\
\dot{I}_j(t)=-\varepsilon\mu(\varepsilon)\partial_{\theta_j}f'(\theta(t),I(t)), \\
\dot{\theta}_l(t)=2\varepsilon A(\theta_1(t),\dots,\theta_d(t))I(t)\cdot e_l+\varepsilon\mu(\varepsilon)\partial_{I_l} f'(\theta(t),I(t)), \\
\dot{\theta}_j(t)=\tilde{\omega}_j+2\varepsilon A(\theta_1(t),\dots,\theta_d(t))I(t)\cdot e_j+\varepsilon\mu(\varepsilon)\partial_{I_j} f'(\theta(t),I(t))
\end{cases}
\end{equation}
where $\tilde{\omega}=(\tilde{\omega}_{d+1},\dots,\tilde{\omega}_n) \in \R^{n-d}$. For some $0<\delta<1$ to be chosen below, let $\tau=\delta\varepsilon^{-1}$. From the first three equations of~\eqref{mouve}, one gets
\[ \max_{1 \leq l \leq d}|I_l(t)-I_l(0)|\leq \delta(4nC_1+\mu(\varepsilon)C')\leq \delta(4nC_1+1), \quad 0 \leq t \leq \tau \]
\[ \max_{d+1 \leq j \leq n}|I_j(t)-I_j(0)|\leq \delta\mu(\varepsilon)C' \leq \delta, \quad 0 \leq t \leq \tau \]
\[ \max_{1 \leq l \leq d}|\theta_l(t)-\theta_l(0)|\leq \delta(8nC_1+\mu(\varepsilon)C')\leq \delta(8nC_1+1), \quad 0 \leq t \leq \tau \]
provided $\mu_0 \leq 1/C'$. From these last estimates, and the estimate~\eqref{min2}, one can choose $\delta$ in terms of $n$, $C_1$ and $\gamma_i^*$ such that 
\begin{equation*}
|\partial_{\theta_i}A(\theta(t))I(t)\cdot I(t)| \geq 2\gamma_i^*, \quad 0 \leq t \leq \tau, 
\end{equation*}
and, assuming $\mu_0 \leq \gamma_i^*/C'$, one obtains
\begin{equation*}
|\varepsilon\partial_{\theta_i}A(\theta_1(t),\dots,\theta_d(t))I(t)\cdot I(t) -\varepsilon\mu(\varepsilon)\partial_{\theta_i}f'(\theta(t),I(t))|\geq \varepsilon\gamma_i^*, \quad 0 \leq t \leq \tau.
\end{equation*}
From the first equation of~\eqref{mouve} this implies
\[ |I_i(\tau)-I_i(0)|\geq \delta\gamma_i^* \]
whereas from the second equation, we recall that
\[ \max_{d+1 \leq j \leq n}|I_j(t)-I_j(0)|\leq \delta\mu(\varepsilon)C', \quad 0 \leq t \leq \tau. \]
This solution of the system associated to $H \circ \Phi$ gives rise to a solution of the system associated to $H$, that we still denote, abusing notations, by $(\theta(t),I(t))$. It satisfies $I(0)=I_0$ and, requiring that $\mu_0 \leq (2C_*)^{-1}\delta\gamma_i^*$ and $\mu_0 \leq (2C_*)^{-1}\delta C'$, one obtains from the estimate~\eqref{estdist} that
\[ |I_i(\tau)-I_i(0)|\geq \delta\gamma_i^*/2  \]
and    
\[ \max_{d+1 \leq j \leq n}|I_j(t)-I_j(0)|\leq \delta\mu(\varepsilon)C'/2, \quad 0 \leq t \leq \tau. \]
Letting $c:=\delta\gamma_i^*/2$ and $C:=\delta C'/2$, this proves the statement.
\end{proof}

\paraga The proof of Theorem~\ref{mainthm2} and Theorem~\ref{mainthm3} are direct consequences from the proof of Theorem~\ref{mainthm1}.

\begin{proof}[Proof of Theorem~\ref{mainthm2}]
We proceed exactly as in Theorem~\ref{mainthm1}, choosing $I_0=e_i$ and thus $\gamma_i^*=-a_i^*/4$. The choice of the sign of $a_i^*$ allows us to ensure that
\[ |I_i(\tau)-I_i(0)|=I_i(\tau)-I_i(0)\geq c \] 
and hence
\[ I_i(\tau)\geq I_i(0)+c=1+c.  \]
Since $I_k(0)=0$ for $1 \leq k \leq n$, $k \neq i$, it is easy to observe that $|I_k(\tau)|<1$ and therefore
\[ |I(\tau)|=I_i(\tau)\geq 1+c=(1+c)|I(0)|. \]
\end{proof}  

\begin{proof}[Proof of Theorem~\ref{mainthm3}]
Once again, we proceed exactly as in Theorem~\ref{mainthm1}, choosing an arbitrary $I_0$ with $|I_0|=1$ and with $\gamma_i^*=-\lambda_i^*/4$. 
\end{proof}  

\paraga To conclude, let us give the proof of Theorem~\ref{mainthm4}, following the arguments of~\cite{Bou14}, which is based on the application of the main theorem of~\cite{Pos82}.

\begin{proof}[Proof of Theorem~\ref{mainthm4}]
In view of the scalings~\eqref{scale1}, it is sufficient to prove the statement for a Hamiltonian $\tilde{H}$ as in~\eqref{Ham2}, defined on the domain $\T^n \times B_3$. We assume that $\mu_0 \leq \mu_*$ so that Proposition~\ref{normalform} can be applied. Then observe it is sufficient to prove the statement for the Hamiltonian $\tilde{H} \circ \Phi$: indeed, $\Phi$ is symplectic and the estimate~\eqref{estdist} implies that it has a Jacobian arbitrarily close to one, hence the measure estimate for the set of invariant tori for $\tilde{H} \circ \Phi$ is equivalent to the measure estimate for the set of invariant tori for $\tilde{H}$.

Let us denote $H=\tilde{H}\circ \Phi$. Using the assumption $(A.4)$ and proceeding as in the proof of Theorem~\ref{mainthm1}, it can be written as
\[ H (\theta,I) = \omega \cdot I +\varepsilon A_0I\cdot I+ \varepsilon\mu(\varepsilon)f'(\theta,I) \]
where $A_0 \in \mathrm{Sym}(n,\R)$ is a non-singular matrix and with
\begin{equation*}
|f'|_{C^{l-1}(\T^n \times B_2)}\leq C'.
\end{equation*}
Up to rescaling time by $\varepsilon$, it is sufficient to prove the statement for the Hamiltonian
\[ \varepsilon^{-1}H (\theta,I) = \varepsilon^{-1}\omega \cdot I +A_0I\cdot I+ \mu(\varepsilon)f'(\theta,I):=H_0(I)+\mu(\varepsilon)H_1(\theta,I). \]
The above Hamiltonian satisfies the assumption of Theorem A in~\cite{Pos82}. Indeed, the Hamiltonian is of class $C^{l-1}$, with $l-1\geq l_0>3n-1$, the integrable part $H_0$ is real-analytic and its Hessian at any point is constantly equal to the non-singular matrix $2A_0$, hence it is bounded by $C_1$ and its inverse is bounded by some uniform positive constant $\tilde{C}_1$. Assuming that $\sqrt{\mu(\varepsilon)} \leq \tilde{\mu}$ where $\tilde{\mu}=\tilde{\mu}(n,C_1,C')$ is some positive constant, the statement follows directly from Theorem A in~\cite{Pos82}. 
\end{proof}

\addcontentsline{toc}{section}{References}
\bibliographystyle{amsalpha}
\bibliography{UnstableTori}

\end{document}